\documentclass[11pt]{article}
\usepackage[utf8]{inputenc}

\usepackage[paperwidth=8.5in,paperheight=11in,portrait,top=1in,bottom=1.in,left=1.15in,right=1.15in]{geometry}
\usepackage{authblk}

\usepackage{amsfonts,amsmath,amssymb,amsthm}
\usepackage{latexsym,mathrsfs,mathtools,bm}
\usepackage{graphicx,subcaption,epsfig,caption,float,xcolor}
\usepackage{enumitem}

\usepackage[hidelinks]{hyperref}
\usepackage{bookmark}

\theoremstyle{plain}
\newtheorem{thm}{Theorem}[section]

\newtheorem{prop}[thm]{Proposition}

\newtheorem{rem}[thm]{Remark}

\numberwithin{equation}{section}
      \def\cI{{\mathcal I}}
      \def\cL{{\mathcal L}}
\def\cM{{\mathcal M}}      
      
\def\cS{{\mathcal S}}   \def\cT{{\mathcal T}}

\title{\bf An algebraic treatment of the Pastro polynomials on the real line}

\renewcommand*{\Affilfont}{\normalsize\small}
\author[1]{Vutha Vichhea Chea}
\author[2]{Luc Vinet}
\author[3]{Meri Zaimi}
\author[4]{Alexei Zhedanov\vspace{.5em}}
\affil[1,2,3]{Centre de Recherches Math\'ematiques, Universit\'e de Montr\'eal, P.O. Box 6128, Centre-ville Station, Montr\'eal (Qu\'ebec), H3C 3J7, Canada. \vspace{.5em}}
\affil[2]{Insitut de valorisation des donn\'ees (IVADO), Montr\'eal (Qu\'ebec), H2S 3H1, Canada. \vspace{.5em}}
\affil[4]{School of Mathematics, Renmin University of China, Beijing 100872, China. \newline\vspace{.9em}}

{
	\makeatletter
	\renewcommand\AB@affilsepx{: \protect\Affilfont}
	\makeatother
	\affil[ ]{E-mail addresses}
	\makeatletter
	\renewcommand\AB@affilsepx{, \protect\Affilfont}
	\makeatother
	\affil[1]{vutha.vichhea.chea@umontreal.ca}
	\affil[2]{luc.vinet@umontreal.ca}
	\affil[3]{meri.zaimi@umontreal.ca}
	\affil[4]{zhedanov@yahoo.com}
}

\begin{document}

\date{\today} 
\maketitle

\begin{abstract}
\noindent The properties of the Pastro polynomials on the real line are studied with the help of a triplet of $q$-difference operators. The $q$-difference equation and recurrence relation these polynomials obey are shown to arise as generalized eigenvalue problems involving the triplet of operators, with the Pastro polynomials as solutions. Moreover, a discrete biorthogonality relation on the real line for the Pastro polynomials is obtained and is then understood using adjoint operators. The algebra realized by the triplet of $q$-difference operators is investigated.      
\end{abstract}

\section{Introduction} \label{sect intro}

The purpose of this paper is to study the bispectrality and biorthogonality of the Pastro polynomials on the real line using a set of $q$-difference operators with bidiagonal actions.   

The Pastro polynomials were introduced in \cite{Pas85} as an infinite family of biorthogonal polynomials on the complex unit circle with respect to the $q$-beta distribution (see also \cite{AI}). They are expressed in terms of a two-parameter $q$-hypergeometric series of type ${}_2\phi_1$. Their biorthogonality in the case where $q$ is a root of unity has been studied in \cite{Zhe98}. The Pastro polynomials belong to the class of Laurent biorthogonal polynomials, which can be characterized by a three-term recurrence relation of the form \cite{Hen-Ros86,Ism-Mas95,Vin-Zhe01,Zhe98} 
\begin{equation}
    P_{n+1}(z) + \mu_n^{(1)}P_n(z) = z(P_n(z) + \mu_n^{(2)}P_{n-1}(z)). \label{rel rec Pastro}
\end{equation} 
In addition to this recurrence relation, they also satisfy a $q$-difference equation that can be recovered from that of the little $q$-Jacobi polynomials through a reparametrization which depends on the degree $n$ \cite{Koekoek}. Therefore, the Pastro polynomials are said to be bispectral, in a generalized sense. 

Past developments have shown that biorthogonal functions with bispectral properties can be understood algebraically in terms of operators in a way analogous to the classical orthogonal polynomials (OPs) of the ($q$-)Askey scheme. Indeed, all classical OPs satisfy both a recurrence relation and a ($q$-)difference or differential equation that can be viewed as two eigenvalue problems for two operators $X$ and $Y$ which act in a tridiagonal fashion either on the degree or the variable \cite{Koekoek}. The OPs then arise as solutions to these eigenvalue problems. The pair of bispectral operators $X,Y$ realizes (a specialization of) the Askey--Wilson algebra \cite{Zhe91} and is known as a Leonard pair in the finite dimensional case \cite{Ter01}. The OPs appear as overlap coefficients in the representation theory of their associated algebra. In the case of biorthogonal functions, one deals instead with a triplet of operators $X,Y,Z$ and the bispectrality is understood in terms of two generalized eigenvalue problems (GEVPs) involving the pairs $X,Y$ and $X,Z$ \cite{Zhe99GEVP}. This algebraic approach has been used for studying several families of bispectral biorthogonal functions: the Hahn and $q$-Hahn rational functions \cite{TVZ21,VZ_Hahn,BGVZ}, the Askey polynomials on the unit circle \cite{VZ_Ask} and the $R_I$ polynomials of Hahn type \cite{VZZ}. The goal of this paper is to adopt a similar approach for understanding the properties of the Pastro polynomials. We will be interested here to the case where the domain is restricted to the real line instead of being the complex plane. This will lead to a different biorthogonality relation for the Pastro polynomials in comparison to previous works.               

The rest of the paper continues as follows. Three $q$-difference operators, that act bidiagonally on the monomial basis, are introduced in Section \ref{sect XYZ}. A generalized eigenvalue problem (GEVP) generated by a pair of these operators is considered in Section \ref{sect Pastro GEVP}. Its polynomial solutions will be shown to be given by the Pastro polynomials. The bispectral properties of the Pastro polynomials are obtained in Section \ref{sect bispec} using the actions of the three $q$-difference operators. Then a discrete biorthogonality relation on the real line for these polynomials is obtained in Section \ref{sect biorth} via two methods. The first method makes use of Baxter's recurrence system \cite{Bax61} while the second method involves a pair of operators and their adjoints. Finally, the algebra realized by the triplet of $q$-difference operators is investigated in Section \ref{sect Pastro alg}. A connection with the $q$-Hahn algebra is made explicit. Section \ref{sect conc} consists of concluding remarks. 

\section{A triplet of $q$-difference operators} \label{sect XYZ}

We begin with the introduction of three $q$-difference operators, crucial for studying the properties of the Pastro polynomials. Throughout this paper, $q$ is assumed generic and real. 

Let $a,b$ be two real parameters. We define the following operators acting on the space $\cM$ of real functions $f(x)$ depending on a real variable $x$:  
\begin{align}
    & \label{def X} X^{(a,b)} = (x-q)T^- + (q-bx)\cI, \\
    & \label{def Y} Y^{(a,b)} = (x-a^{-1}q)T^+ + (a^{-1}q-b^{-1}x)\cI, \\
    & \label{def Z} Z^{(a,b)} = (1-qx^{-1})T^- + (qx^{-1}-b)\cI.
\end{align}
The notations $T^\pm$ and $\cI$ are used to denote the $q$-shift operators and the identity operator respectively. They are defined as
\begin{equation}
    T^\pm f(x) = f(q^{\pm1}x), \quad \cI f(x) = f(x).
\end{equation}
For convenience, we will omit the dependencies of the triplet of operators on the parameters $a,b$ and write $X,Y,Z$ instead.
From the explicit expressions \eqref{def X} and \eqref{def Z}, it is obvious that
\begin{equation}
    Z = x^{-1}X. \label{ZX}
\end{equation}

In what follows, we will consider more specifically the subspace of functions $f(x)\in\cM$ which are polynomials. A basis of interest for this subspace is the monomial basis
\begin{equation}
    \varphi_n(x) = x^n,
\end{equation}
where $n$ is any non-negative integer. All operators $X$, $Y$ and $Z$ act bidiagonally on this basis:
\begin{align}
    & \label{act X} X\varphi_n(x) = (q^{-n}-b)\varphi_{n+1}(x) + q(1-q^{-n})\varphi_n(x), \\
    & \label{act Y} Y\varphi_n(x) = (q^n-b^{-1})\varphi_{n+1}(x) +a^{-1}q(1-q^n)\varphi_n(x), \\
    & \label{act Z} Z\varphi_n(x) = (q^{-n}-b)\varphi_{n}(x) + q(1-q^{-n})\varphi_{n-1}(x).
\end{align}
It is observed that the operators $X,Y,Z$ raise the degree of any polynomial by at most one, and hence belong to the class of generalized Heun operators \cite{BVZ,Vin-Zhe19}.

\section{Pastro polynomials as solutions of a generalized eigenvalue problem} \label{sect Pastro GEVP}

In this section, we demonstrate how the Pastro polynomials arise as solutions of a GEVP built from the pair of operators $X$ and $Y$. We use the standard notations for the $q$-hypergeometric function ${}_r\phi_s$ and the $q$-Pochhammer symbol $(z;q)_n$, see \cite{GR,Koekoek}.

\begin{prop}
	The normalized polynomial solutions $P_n(x)$ of the GEVP
	\begin{equation}
    	YP_n(x) = \lambda_nXP_n(x) \label{GEVP XY} 
	\end{equation}
	are given by the Pastro polynomials  
	\begin{equation}
    	P_n(x) = P_n(x;a,b;q) = \frac{(a^{-1}bq^{1-n};q)_n(q;q)_n}{(q^{-n};q)_n(b;q)_n} \ _2\phi_1\left({q^{-n},b\atop a^{-1}bq^{1-n}};q,x\right) \label{def Pastro} 
	\end{equation}
with eigenvalues
\begin{equation}
    \lambda_n = -b^{-1}q^{n}, \label{lambda}
\end{equation}
for $n=0,1,2,\dots$
\end{prop}
\begin{proof}
	Any polynomial solution to equation \eqref{GEVP XY} can be expanded in the basis of the monomials $\varphi_k(x)$. Therefore, one can take
\begin{equation}
    P_n(x) = \sum_{k=0}^{n} C_{n,k} \ \varphi_k(x), \label{exp Pastro} 
\end{equation}
where $n$ is a non-negative integer and $C_{n,k}$ are some coefficients. The expansion \eqref{exp Pastro} and the bidiagonal actions \eqref{act X} and \eqref{act Y} of $X$ and $Y$ can be used to develop each side of equation \eqref{GEVP XY} in terms of the basis functions $\varphi_k(x)$. Equalizing the coefficients of the terms proportional to $\varphi_{n+1}(x)$, one finds that the associated eigenvalue $\lambda_n$ must satisfy equation \eqref{lambda}. The other terms lead to the following two-term recurrence relation for the coefficients $C_{n,k}$, for $k=0,1,\dots,n-1$:
\begin{equation}
    (1-q^{k-n})(1-bq^k)C_{n,k} = (1-a^{-1}bq^{k+1-n})(1-q^{k+1})C_{n,k+1}. \label{two-term rec}
\end{equation}
Note that the explicit expression \eqref{lambda} of the eigenvalue $\lambda_n$ has been used to obtain \eqref{two-term rec}. The solution to this recurrence relation is found to be
\begin{equation}
    C_{n,k} = \frac{(q^{-n};q)_k(b;q)_k}{(a^{-1}bq^{1-n};q)_k(q;q)_k}C_{n,0}. \label{Cnk}
\end{equation}
The solution \eqref{def Pastro} then follows by substituting \eqref{Cnk} into the expansion \eqref{exp Pastro}, and by choosing the normalization constant $C_{n,0}$ so as to provide monic condition \begin{equation}
    P_n(x) = x^n + O(x^{n-1}).
\end{equation}
\end{proof}
Using the explicit expression \eqref{def Pastro} and the actions \eqref{act X}, \eqref{act Y} and \eqref{act Z}, it can be shown that the triplet of operators $X,Y,Z$ have the effect of producing a $q$-shift on the paramater $b$ when acting on the Pastro polynomials $P_n(x;a,b;q)$. More precisely, we have
\begin{align}
    & \label{ctg rel X} XP_n(x;a,b;q) = q^{-n}(1-bq^n)xP_n(x;a,bq;q), \\
    & \label{ctg rel Y} YP_n(x;a,b;q) = -b^{-1}(1-bq^n)xP_n(x;a,bq;q), \\
    & \label{ctg rel Z} ZP_n(x;a,b;q) = q^{-n}(1-bq^n)P_n(x;a,bq;q).
\end{align}
\begin{rem}
	The actions \eqref{ctg rel X}--\eqref{ctg rel Z} can be written in terms of the transformation 
\begin{equation}
	\cS : b \mapsto bq	 \label{transf S}
\end{equation}
by replacing the polynomials $P_n(x;a,bq;q)$ which appear on the RHSs by $\cS P_n(x;a,b;q)$.
\end{rem}

\section{Bispectrality} \label{sect bispec}

Let us now show how the bispectrality of the Pastro polynomials is expressed through  two GEVPs: a $q$-difference equation in the variable $x$ with eigenvalue depending on the degree $n$ and a recurrence relation in the degree $n$ with eigenvalue depending on the variable $x$.

\subsection{$q$-Difference equation}

Consider again the GEVP \eqref{GEVP XY}: $YP_n = \lambda_nXP_n$. We have shown in Section \ref{sect Pastro GEVP} that the Pastro polynomials $P_n(x)$ arise as solutions to this GEVP with eigenvalues $\lambda_n = -b^{-1}q^{n}$. Observe that this is consistent with the actions \eqref{ctg rel X} and \eqref{ctg rel Y} of the operators $X$ and $Y$ on $P_n(x)$. Using the defining expressions \eqref{def X} and \eqref{def Y} of the operators $X$ and $Y$ in the GEVP \eqref{GEVP XY}, one gets
\begin{equation}
    (x-a^{-1}q)P_n(qx) + (a^{-1}q-b^{-1}x)P_n(x)=-b^{-1}q^{n}\left((x-q)P_n(q^{-1}x) + (q-bx)P_n(x)\right), \label{diff eq Pastro}
\end{equation}
which is a second order $q$-difference equation for the Pastro polynomials $P_n(x)$.

\subsection{Recurrence relation}

We already know that the operators $X$, $Y$ and $Z$ all act bidiagonally in the monomial basis $\varphi_n(x)$. Consider now the action of these operators on the basis of the Pastro polynomials $P_n(x)$. From the contiguity relations of the $_2\phi_1$ series in particular, one finds that the actions \eqref{ctg rel X} and \eqref{ctg rel Z} can be re-expressed as follows:
\begin{align}
    & \label{act rec X} XP_n(x) = q^{-n}(1-bq^{n})P_{n+1}(x) + q(1-a^{-1}bq^{-n})P_n(x), \\
    & \label{act rec Z} ZP_n(x) = q^{-n}(1-bq^{n})P_{n}(x) + \frac{bq(1-q^{-n})(1-aq^{n-1})}{a(1-bq^{n-1})}P_{n-1}(x).
\end{align}
Note that the action of $Y$ on the basis $P_n(x)$ is deduced from \eqref{act rec X} using the GEVP \eqref{GEVP XY}.
Applying both sides of equality \eqref{ZX} on the polynomials $P_n(x)$, we get
\begin{equation}
    ZP_n(x) = x^{-1}XP_n(x). \label{XZPn}
\end{equation}
Then, using the actions \eqref{act rec X} and \eqref{act rec Z} of the operators $X$ and $Z$, equation \eqref{XZPn} can be written as
\begin{equation}
    P_{n+1}(x) + \mu_n^{(1)}P_n(x) = x\left(P_n(x) + \mu_n^{(2)}P_{n-1}(x)\right), \label{same rec rel Pastro}
\end{equation}
where the coefficients $\mu_n^{(1)}$ and $\mu_n^{(2)}$ are given by
\begin{equation}
    \mu_n^{(1)} = -\frac{q(b-aq^n)}{a(1-bq^n)}, \quad \mu_n^{(2)} = -\frac{bq(1-q^n)(1-aq^{n-1})}{a(1-bq^n)(1-bq^{n-1})}. \label{mu}
\end{equation}
Equation \eqref{same rec rel Pastro} is a three-term recurrence relation for the Pastro polynomials $P_n(x)$. Note the presence of the term proportional to $P_{n-1}(x)$ with a factor $x$ which does not appear for classical OPs but which is expected for Laurent biorthogonal polynomials \cite{Hen-Ros86,Ism-Mas95,Zhe98}.

We thus see that both the $q$-difference equation \eqref{diff eq Pastro} and the recurrence relation \eqref{same rec rel Pastro} of the Pastro polynomials can be understood through two GEVPs involving the pairs of operators $X,Y$ and $X,Z$; this is how the triplet $X,Y,Z$ accounts for the bispectrality of the Pastro polynomials.   Note that in order to view \eqref{XZPn} as a GEVP, one should define $X$ and $Z$ as operators acting on the degree $n$ using \eqref{act rec X} and \eqref{act rec Z}.

\section{Biorthogonality} \label{sect biorth}

This section is devoted to the proof of the following result.
\begin{prop}\label{prop:biorth}
	Let $N$ be a positive integer and take the parameter $a$ to be
	\begin{equation}
    a = q^{1-N}. \label{restric 1}
\end{equation}
	Define moreover the following $q$-grid on the real line:
\begin{equation}
    x_s = q^{s+1}, \quad s = 0, 1, \dots, N-1. \label{q-grid}
\end{equation} 
	 Then, the Pastro polynomials $P_n(x;a,b;q)$ satisfy the biorthogonality relation 
	 \begin{equation}
    \sum_{s=0}^{N-1}w_sP_n(x_s)R_m(x_s) = h_n\delta_{mn}, \quad m,n=0,1,\dots,N-1, \label{rel biorth2}
\end{equation}
where the weight function $w_s$ and normalization constant $h_n$ are given by
\begin{align}
    &w_s = \frac{1}{(b;q)_{N-1}}\frac{(q^{1-N};q)_s}{(q;q)_s}(bq^{N-1})^s, \label{full eq weight} \\
    &h_n = \frac{(a;q)_n(q;q)_n}{(ab^{-1}q;q)_n(b;q)_n}, \label{full norm h} 
\end{align}
and where the normalized biorthogonal partners $R_n(x)$ are given by the Laurent polynomials
\begin{equation}
    R_n(x) = R_n(x;a,b;q) = \frac{(q^{-n};q)_n(bq^{-1};q)_n}{(a^{-1}bq^{-n};q)_n(q;q)_n} \ _2\phi_1\left({q^{-n},ab^{-1}q\atop b^{-1}q^{2-n}};q,\frac{q^2}{ax}\right). \label{full poly Q}
\end{equation}
\end{prop}
In Subsection \ref{ssec:Baxter}, the discrete biorthogonality relation for the Pastro polynomials on the real line as described in Proposition \ref{prop:biorth} will be obtained in a systematic manner by making use of a recurrence system proposed by G. Baxter. Then, with hindsight from the known result, the case $m\neq n$ of this biorthogonality relation will be recovered in Subsection \ref{ssec:operators} using the pair of $q$-difference operators $X,Y$ and the properties of their adjoints.  

\subsection{Derivation from recurrence system} \label{ssec:Baxter}

\quad As mentioned in the introduction, the Pastro polynomials are an example of polynomials satisfying a biorthogonality relation on the unit circle. In consequence, they should obey the following recurrence system proposed by G. Baxter \cite{Bax61}:
\begin{align}
    & \label{rec sys 1} P_{n+1}(x) = xP_n(x) - \alpha_nx^nQ_n(x^{-1}), \\
    & \label{rec sys 2} Q_{n+1}(x) = xQ_n(x) - \beta_nx^nP_n(x^{-1}), \\
    & \label{rec sys 3} P_0(x) = Q_0(x) = 1,
\end{align}
where $\alpha_n$ and $\beta_n$ are the recurrence coefficients. This system involves two sets of polynomials, namely the Pastro polynomials $P_n(x)$ and some polynomials $Q_n(x)$. We can define the constant $h_n$ as follows
\begin{equation}
    h_n = \prod_{k=0}^{n-1}(1-\alpha_k\beta_k), \quad h_0 = 1. \label{norm h}
\end{equation}
For $N$ any positive integer, consider the conditions
\begin{align}
    &h_n \neq 0, \quad h_N = 0, \quad \text{for } n=0,1,\dots,N-1, \label{cond 1} \\
    &P_N(x_s) = 0, \quad P_N'(x_s) \neq 0, \quad \text{for } s=0,1,\dots,N-1, \label{cond 2}
\end{align}
where $x_s$ are all the zeros of the polynomial $P_N(x)$. It has been shown in \cite{Zhe99} that if conditions \eqref{cond 1} and \eqref{cond 2} are satisfied in addition of equations \eqref{rec sys 1}--\eqref{rec sys 3}, then the following biorthogonality relation holds for $m,n = 0, 1, ..., N-1$
\begin{equation}
    \sum_{s=0}^{N-1}w_sP_n(x_s)Q_m(x_s^{-1}) = h_n\delta_{mn}, \label{rel biorth}
\end{equation}
where the normalization constant $h_n$ is given in \eqref{norm h} and the weight function $w_s$ is
\begin{equation}
    w_s = \frac{h_{N-1}}{P_N'(x_s)Q_{N-1}(x_s^{-1})}. \label{eq weight}
\end{equation}
Our goal for the remainder of this subsection is to find explicit expressions for the objects which appear in \eqref{rel biorth}, hence leading to the biorthogonality relation of Proposition \ref{prop:biorth}.

We begin by finding the recurrence coefficients $\alpha_n$ and $\beta_n$ of Baxter's system \eqref{rec sys 1}--\eqref{rec sys 3}. The first relation \eqref{rec sys 1} can be solved for $Q_n(x^{-1})$, giving
\begin{equation}
    Q_n(x^{-1}) = \frac{xP_n(x) - P_{n+1}(x)}{\alpha_nx^n}. \label{poly Q}
\end{equation}
Substituting \eqref{poly Q} into the second relation \eqref{rec sys 2}, one gets the following three-term recurrence relation
\begin{equation}
    P_{n+1}(x) - \frac{\alpha_n}{\alpha_{n-1}}P_n(x) = x \left( P_n(x) + \left(\alpha_n\beta_{n-1} - \frac{\alpha_n}{\alpha_{n-1}}\right) P_{n-1}(x)\right). \label{rel rec Pastro 2}
\end{equation}
Comparing \eqref{rel rec Pastro 2} with the recurrence relation \eqref{same rec rel Pastro} of the Pastro polynomials, one deduces
\begin{align}
    & \label{rec coeff alpha} \alpha_n = -\alpha_{n-1}\mu_n^{(1)}, \\
    & \label{rec coeff beta} \beta_n = \frac{\mu_{n+1}^{(2)} - \mu_{n+1}^{(1)}}{\alpha_{n+1}},
\end{align}
where the coefficients $\mu_n^{(1)}$ and $\mu_n^{(2)}$ are given in \eqref{mu}. 
Equation \eqref{rec coeff alpha} is a two-term recurrence relation for the coefficient $\alpha_n$ which can be solved up to the factor $\alpha_0$. This factor can be fixed by taking $n=0$ in \eqref{poly Q} and then using the initial condition $Q_0(x)=1$ from relation \eqref{rec sys 3} together with the explicit expression \eqref{def Pastro} of the Pastro polynomials. The result is that the coefficient $\alpha_n$ is given by
\begin{equation}
    \alpha_n = -(a^{-1}bq)^{n+1}\frac{(ab^{-1};q)_{n+1}}{(b;q)_{n+1}}. \label{full rec coeff alpha}
\end{equation}
For the coefficient $\beta_n$, by substituting \eqref{full rec coeff alpha} into \eqref{rec coeff beta}, one obtains after some simplifications
\begin{equation}
    \beta_n = -(a^{-1}b)^{-n-1}\frac{(bq^{-1};q)_{n+1}}{(ab^{-1}q;q)_{n+1}}. \label{full rec coeff beta}
\end{equation}
The explicit expressions of the normalization constant $h_n$ and of the biorthogonal partners $R_n(x):=Q_n(x^{-1})$ can now be straightforwardly computed using \eqref{full rec coeff alpha} and \eqref{full rec coeff beta} in \eqref{norm h} and \eqref{poly Q}. One arrives precisely at the results \eqref{full norm h}
and \eqref{full poly Q}.

It is observed from \eqref{full norm h} that the condition \eqref{cond 1} is only satisfied when $a = q^{1-N}$, which is equation \eqref{restric 1}. Under this restriction and with the help of the $q$-binomial theorem \cite{Koekoek}
\begin{equation}
    _1\phi_0 \ (q^{n};-;q,z) = (q^{-n}z;q)_n, \label{bin trm}
\end{equation}
the zeros $x_s$ of $P_N(x)$ are then found to be given by equation \eqref{q-grid} and to be simple, thereby satisfying condition \eqref{cond 2}. We are finally ready to calculate the weight function $w_s$. Still under the restriction $a = q^{1-N}$, by using its given form \eqref{eq weight}, a straightforward calculation leads to the expression \eqref{full eq weight}. Note that the weight function $w_s$ is normalized, that is $\sum_{s=0}^{N-1}w_s = 1$, as can be verified with the help of the $q$-binomial theorem \eqref{bin trm}.  This ends the proof of Proposition \ref{prop:biorth}.

\begin{rem}
	Under the transformation
\begin{equation}
    \begin{aligned}
    \cT : \ & b \mapsto b^{-1}q^{2-N}, \\
    & s \mapsto N-s-1,
    \end{aligned} \label{transf T}
\end{equation}
the weight function $w_s$ remains unchanged.
\end{rem}

\subsection{Derivation from operators} \label{ssec:operators}

We will now recover the biorthogonal partners $R_n(x)$ of the Pastro polynomials that are given in \eqref{full poly Q} using the $q$-difference operators $X$ and $Y$. We will restrict for the remainder of the present subsection the parameter $a$ to be as in \eqref{restric 1}, with $N$ any positive integer. We will also restrain the variable $x$ to take values only on the $q$-grid $x_s$ given in \eqref{q-grid}.

Let us equip the space of real Laurent polynomials $\cL$ involving the parameter $b$ and defined on the $q$-grid $x_s$ with the scalar product
\begin{equation}
    (f(x_s),g(x_s))_{(b)} = \sum_{s=0}^{N-1} w_s^{(b)}f(x_s)g(x_s), \label{scalar prod}
\end{equation}
where $f,g \in \cL$ and $w_s^{(b)}$ is the weight function \eqref{full eq weight}. The dependency of the scalar product and the weight function on the parameter $b$ will be omitted in what follows. For a given operator $W$ acting on the space $\cL$, its adjoint $W^*$ is defined to be the operator that satisfies
\begin{equation}
    (Wf(x_s),g(x_s)) = (f(x_s),W^*g(x_s)). \label{adj id}
\end{equation}
Consider now the adjoint of the GEVP \eqref{GEVP XY}:
\begin{equation}
    Y^*P_n^*(x_s) = \lambda_n^*X^*P_n^*(x_s), \label{adj GEVP XY}
\end{equation}
where $\lambda_n^*$ are the adjoint eigenvalues coinciding with $\lambda_n = -b^{-1}q^{n}$ and $P_n^*(x_s)$ stands for the solution of \eqref{adj GEVP XY}. It is known from \cite{Zhe99GEVP} (see also proof of Proposition 4 in \cite{BGVZ} for instance) that
\begin{equation}
	(P_n(x_s),X^*P_m^*(x_s)) = 0 \quad \text{for } m\neq n,
\end{equation}
meaning that the functions
\begin{equation}
    R_n(x_s) = X^*P_n^*(x_s) \label{eq R}
\end{equation}
are biorthogonal partners for the Pastro polynomials $P_n(x_s)$. We shall now obtain the explicit form of $R_n(x_s)$ by making use of the GEVP \eqref{GEVP XY} and the properties of the adjoint operators $X^*$ and $Y^*$ contained in the adjoint GEVP \eqref{adj GEVP XY}. Before we begin, note that under the constraint \eqref{restric 1} and the restriction of $x$ to the $q$-grid \eqref{q-grid}, the operators $X$ and $Y$ become
\begin{align}
    & \label{restric def X} X = q(q^s-1)T^- + q(1-bq^s)\cI, \\
    & \label{restric def Y} Y = (q^{s+1}-q^N)T^+ + (q^N-b^{-1}q^{s+1})\cI,              
\end{align}
with $T^{\pm}f(x_s) = f(x_{s \pm 1})$ for $s=0,1,\dots,N-1$, using the convention $f(x_{-1})=f(x_{N})=0$. Also note that equation \eqref{ctg rel X} translates to
\begin{equation}
    XP_n(x_s) =(q^{-n}-b)q^{s+1}\cS P_n(x_s), \label{restric ctg rel X}
\end{equation}
where we recall that the transformation $\cS$ is defined in \eqref{transf S}.

Let us start by finding the expressions of the adjoint shift operators $(T^{\pm})^*$. This is easily done with the use of the identity \eqref{adj id} and the scalar product \eqref{scalar prod}. The result is
\begin{align}
    & (T^{+})^* = \frac{w_{s-1}}{w_s} T^- = \frac{q(1-q^s)}{b(q^N-q^s)}T^-, \label{adj Tp} \\
    & (T^{-})^* = \frac{w_{s+1}}{w_s} T^+  = \frac{b(q^N-q^{s+1})}{q(1-q^{s+1})}T^+. \label{adj Tm}
\end{align}
Using the restricted forms \eqref{restric def X}--\eqref{restric def Y} of the operators $X,Y$ and the results \eqref{adj Tp}--\eqref{adj Tm}, one obtains the adjoint operators $X^*$ and $Y^*$:
\begin{align}
    & X^* = b(q^{s+1}-q^N)T^+ + q(1-bq^s)\cI, \\
    & Y^* = b^{-1}q(q^s-1)T^- + (q^N-b^{-1}q^{s+1})\cI.
\end{align}
Consider now the transformation
\begin{equation}
    \begin{aligned}
    \tau : \ & b \mapsto b^{-1}q^{1-N}, \\
    & s \mapsto N-s-1, \\
    & T^{\pm} \mapsto T^{\mp}.
    \end{aligned} \label{tranf tau}
\end{equation}
One can observe that the action of $\tau$ on $b$ and $s$ is equivalent to the action of $\cS \circ \cT$, see equations \eqref{transf S} and \eqref{transf T}, and that $\tau^2$ is the identity transformation. Note that $\tau$ has been chosen so that for any $q$-difference operator $W$ and any function $f\in\cL$, the following relation holds
\begin{equation}
	\tau(Wf(x_s))=\tau(W)\tau(f(x_s)).
\end{equation}
By applying the transformation $\tau$ to the operators $X$ and $Y$ given in \eqref{restric def X} and \eqref{restric def Y}, resulting in the transformed operators $\tau(X)$ and $\tau(Y)$, a relation with the adjoint operators $X^*$ and $Y^*$ shows up:
\begin{align}
    &X^*= -bq^s\tau(X), \label{act trans X} \\
    &Y^*= -b^{-1}q^{s+1-N}\tau(Y). \label{act trans Y}
\end{align}
On one hand, relations \eqref{act trans X} and \eqref{act trans Y} can be used to rewrite the adjoint GEVP \eqref{adj GEVP XY} with eigenvalue $\lambda_n^* =\lambda_n= -b^{-1}q^{n}$ as
\begin{equation}
    \tau(Y)P_n^*(x_s)=-bq^{n+N-1}\tau(X)P_n^*(x_s). \label{trans GEVP XY}
\end{equation}
On the other hand, one can apply the transformation $\tau$ on both sides of the GEVP \eqref{GEVP XY}. The resulting equation can be written as
\begin{equation}
	\tau(Y)\tau(P_n(x_s)) = \tau(\lambda_n)\tau(X)\tau(P_n(x_s)). \label{trans GEVP XY 2}
\end{equation}
Using again the expression of the eigenvalues $\lambda_n = -b^{-1}q^{n}$ and the transformation \eqref{tranf tau}, it is found that
\begin{equation}
    \tau(\lambda_n) = -bq^{n+N-1},
\end{equation}
which is the factor appearing in \eqref{trans GEVP XY}. Therefore, both equations \eqref{trans GEVP XY} and \eqref{trans GEVP XY 2} are GEVPs for the transformed operators $\tau(X)$ and $\tau(Y)$ with the same eigenvalues. By comparison of these two equations, it follows that
\begin{equation}
    P_n^*(x_s)=\nu_n\tau(P_n(x_s)),  \label{adPtauP}
\end{equation}
where $\nu_n$ is some normalization which does not depend on $s$.
%Applying the transformation $\tau$ on $P_n(x_s)$, one finds from \eqref{adPtauP} that the solutions of the adjoint GEVP \eqref{adj GEVP XY} are
%\begin{equation}
%    P_n^*(x_s) = P_n^*(q^{s+1}) =  \nu_n P_n(q^{N-s};q^{1-N},b^{-1}q^{1-N};q). \label{trans sol Pastro}
%\end{equation}
It is now possible to compute the biorthogonal partners $R_n(x_s)$. Combining equations \eqref{eq R}, \eqref{act trans X} and \eqref{adPtauP}, one gets 
\begin{equation}
    R_n(x_s) = -bq^s\nu_n \tau(X)\tau(P_n(x_s))=-bq^s\nu_n \tau(XP_n(x_s)). \label{restric eq R}
\end{equation}
Recall the transformation $\cT$ defined in \eqref{transf T}. Applying $\tau$ on both sides of equation \eqref{restric ctg rel X} and using the fact that $\tau \circ \cS = \cT$ yields
\begin{equation}
    \tau(XP_n(x_s)) = q^{-s}(q^{N-n}-b^{-1}q)\cT P_n(x_s). \label{act tY Pastro}
\end{equation}
After substituting \eqref{act tY Pastro} into the RHS of \eqref{restric eq R} and absorbing all factors which do not depend on $s$ in $\nu_n$, one finally finds the expression of the biorthogonal partners to be
\begin{equation}
    R_n(x_s)= \nu_n\cT P_n(x_s)=\nu_nP_n(q^{N-s};q^{1-N},b^{-1}q^{2-N};q). \label{RnPn}
\end{equation}
It can be verified that the function $R_n(x_s)$ given in \eqref{RnPn} corresponds up to a normalization factor to the function $R_n(x;a,b;q)$ given in \eqref{full poly Q} under the constraint \eqref{restric 1} for $a$ and the restriction of $x$ on the $q$-grid \eqref{q-grid}. Therefore, we have recovered the biorthogonal partners of the Pastro polynomials using $q$-difference operators and GEVPs.

\section{Pastro algebra} \label{sect Pastro alg}

In the case of the classical orthogonal polynomials of the ($q$-)Askey scheme, it is known that the associated pair of bispectral operators satisfy an algebra of Askey--Wilson type \cite{Zhe91}. In this section, we identify the algebra generated by the triplet of bispectral operators of the Pastro polynomials; we will call this the Pastro algebra. Connections with the Askey--Wilson algebra and its specialization the $q$-Hahn algebra will be made.       

By direct computation using definitions \eqref{def X}--\eqref{def Z}, it is found that the operators $X,Y,Z$ obey the following relations
\begin{align}
    & \label{com rel XY} qXY - YX =q(q-1)\left( a^{-1}X + Y \right), \\
    & \label{com rel YZ} qYZ - ZY = (q-1)\left( -b^{-1}q^{-1}(1+q)X - bY + a^{-1}qZ + a^{-1}b^{-1}(1-b)(a-bq)\right), \\
    & \label{com rel ZX} qZX - XZ = (q-1)\left(-bX + qZ\right). 
\end{align}
Under the affine transformations
\begin{align}
    & X \mapsto \frac{a}{bq(q-1)}X - \frac{a}{b(q-1)}, \\
    & Y \mapsto bq^{-1}Y - a^{-1}b, \\
    & Z \mapsto -a^{-1}Z - a^{-1}b,
\end{align}
the relations \eqref{com rel XY}--\eqref{com rel ZX} become
\begin{align}
    & qXY - YX = 1, \label{rel Pastro 1} \\
    & qYZ - ZY = \alpha_1X + \alpha_2, \label{rel Pastro 2} \\
    & qZX - XZ = 1, \label{rel Pastro 3}
\end{align}
where 
\begin{equation}
    \alpha_1 = a^{-2}bq^{-1}(q-1)^2(q+1), \quad \alpha_2 = a^{-2}q^{-1}(q-1)(ab+aq+bq).
\end{equation}
The Pastro algebra can then be observed to be a special case of the Askey--Wilson algebra \cite{Zhe91}. We comment on this below.

The usual Casimir element of the Askey--Wilson algebra (see \cite{Vin-Zhe10} for instance) takes the following form in the specialization \eqref{rel Pastro 1}--\eqref{rel Pastro 3}: 
\begin{equation}
    Q = (q^{-2}-1)XYZ + q^{-1}\alpha_1X^2 + q^{-1}(q^{-1}+1)(\alpha_2X + q^{-1}Y + Z). \label{casimir}
\end{equation}
In other words, the element $Q$ given above is central in the Pastro algebra, as can be verified directly using relations \eqref{rel Pastro 1}--\eqref{rel Pastro 3}.

\subsection{The Askey--Wilson algebra and biorthogonal functions}

Let us recall the $\mathbb{Z}_3$ presentation of the Askey--Wilson algebra \cite{WZ, IT, CFGPRV}:
\begin{align} 
    & qXY - YX = \beta_1 Z + \beta_2, \label{rel AW 1} \\
    & qYZ - ZY = \alpha_1X + \alpha_2, \label{rel AW 2} \\
    & qZX - XZ = \delta_1Y + \delta_2, \label{rel AW 3}
\end{align}

These relations can be standardized by affine transformations in non-singular parametric situations. Manifestly the algebra associated to the Pastro polynomials and defined by \eqref{rel Pastro 1}--\eqref{rel Pastro 3} is a special case of the Askey--Wilson algebra with $\beta_1=0, \beta_2=1$ and $\delta_1=0$,  $\delta_2=1$. A first significant observation is that the Askey--Wilson algebra (and its degenerations) does not only provide algebraic interpretations of the orthogonal polynomials of the Askey scheme but also, as is the case here, of biorthogonal polynomials. 

As we shall indicate in the next subsection, the situation where either $\beta_1$ or $\delta_1$ is equal to $0$ (but not both) corresponds to the algebra describing the bispectrality of the $q$-Hahn polynomials. It is hence tempting to ask to what orthogonal polynomials does correspond the algebra defined by \eqref{rel Pastro 1}--\eqref{rel Pastro 3} where both $\beta_1$ and $\delta_1$ are equal to $0$. It is interesting to realize that this is an ill-founded question and that there are no finite families of orthogonal polynomials whose bispectrality is encoded in this algebra. Indeed, discussing orthogonal polynomials in this algebraic framework, one looks at representations in bases where one of the generator is diagonal. (Recall that we have considered here in contradistinction generalized eigenbases.) On the one hand, the relations \eqref{rel Pastro 1} and \eqref{rel Pastro 3} are each recognized as defining a $q$-oscillator algebra. It is known \cite{MS,TVZ17}, that in any finite-dimensional irreducible representation of this algebra, the generators have a  spectrum of exponential type $q^{\pm{k}}$ and such must hence be the case for $X,Y,Z$ in view of \eqref{rel Pastro 1} and \eqref{rel Pastro 3}. On the other hand, given \eqref{rel Pastro 2}, when $\alpha_1\neq 0$, the representation theory of the Askey--Wilson algebra \cite{Zhe91} shows that in finite dimensions $Z$ must have a spectrum of the form $c_1q^k+c_2q^{-k}$ where the constants $c_1$ and $c_2$ are such that $c_1c_2\neq0$. This is in contradiction with the preceding assertion. Note however that this conundrum does not arise when $\alpha_1=0$ or in the case of the $q$-Hahn algebra to be discussed below where $\beta_1$ or $\delta_1$ is $\neq 0$ and there is only one $q$-oscillator relation.

%Relations \eqref{rel Pastro 1} and \eqref{rel Pastro 3} correspond more precisely to the defining relations of the $q$-oscillator algebra. As a consequence, in any finite-dimensional irreducible representation, the operators $X,Y,Z$ must have a spectrum of exponential form $q^{k}$ \textnormal{\cite{MS,TVZ17}}, which is not compatible with the hyperbolic spectrum $q^k+q^{-k}$ of the Askey--Wilson relation \eqref{rel Pastro 2} \textnormal{\cite{Zhe91}}. This shows that the Askey--Wilson algebra encodes more than the bispectral properties of the classical OPs of the ($q$-)Askey scheme. 
%\end{rem}

\subsection{Connection with the $q$-Hahn algebra} \label{sect Hahn alg}

One distinctive feature of the Pastro algebra is that it contains the $q$-Hahn algebra. The $q$-Hahn algebra is associated to the $q$-Hahn polynomials and is defined by three generators $A$, $B$ and $C$ subject to the relations
\begin{align}
    & \label{com rel AB} qAB - BA = C + \gamma_1, \\
    & \label{com rel BC} qBC - CB = \gamma_2, \\
    & \label{com rel CA} qCA - AC = \gamma_3B + \gamma_4,
\end{align}
where $\gamma_i$ for $i = 1,2,3,4$ are arbitrary real constants. In the realization of the $q$-Hahn algebra by bispectral operators, these coefficients $\gamma_i$ take explicit values involving the parameters of the $q$-Hahn polynomials. In what follows, it will be useful to view the coefficients $\gamma_i$ as central elements instead.

Let us now describe the embedding of the $q$-Hahn algebra into the Pastro algebra by introducing the linear pencil
\begin{equation}
    L = X + \mu Y,
\end{equation}
where $\mu$ is an arbitrary real parameter. We define an operator $M$ in terms of the operators $L$ and $Z$ as follows:
\begin{equation}
    M = qLZ - ZL - \mu\alpha_2 - 1.
\end{equation}
It can be verified that the elements $L,M,Z$ together with the central element $Q$ given in \eqref{casimir} generate an algebra with the following relations:
\begin{align}
    & qLZ - ZL= M + \mu\alpha_2 + 1, \label{com rel LZ} \\
    & qZM - MZ = \mu\alpha_1, \label{com rel ZM} \\
    & qML - LM = \mu q^{-1}(q+1)^2(q-1)Z - \mu q^{2}(q-1)Q + \mu^2\alpha_1. \label{com rel ML}
\end{align}
Comparing \eqref{com rel AB}--\eqref{com rel CA} with \eqref{com rel LZ}--\eqref{com rel ML}, it is then seen that the $q$-Hahn algebra is embedded in the Pastro algebra under the mappings
\begin{align}
	& \gamma_1 \mapsto \mu\alpha_2 + 1, \quad \! \gamma_2 \mapsto \mu\alpha_1, \quad \! \gamma_3 \mapsto \mu q^{-1}(q+1)^2(q-1), \quad \! \gamma_4 \mapsto -\mu q^{2}(q-1)Q + \mu^2\alpha_1,\\
    & A \mapsto L, \quad B \mapsto Z, \quad C \mapsto M.
\end{align}

\section{Conclusion} \label{sect conc}

In summary, we have introduced a triplet of $q$-difference operators that plays a role similar to the pair of bispectral operators of the classical OPs, but for biorthogonal polynomials instead. We have then shown through two GEVPs how these operators provide a description of the bispectral property of the Pastro polynomials. We have also shown how to obtain a discrete biorthogonality relation of the Pastro polynomials on the real line using two methods; the first is by exploiting a recurrence system proposed by G. Baxter and the second is purely with the use of a pair of operators and the properties of their adjoints. The underlying algebra of the Pastro polynomials, called Pastro algebra, was then obtained and observed to be a special case of the Askey--Wilson algebra. A connection with the $q$-Hahn algebra has been established by showing its embedding into the Pastro algebra. 

One natural direction for future research would be to consider the Pastro algebra from a more abstract point of view and study its representation theory, in a similar manner as in \cite{VZ_Hahn,VZ_Ask}. It is expected that the Pastro biorthogonal polynomials will appear as overlap coefficients between bases associated to eigenvalue or generalized eigenvalue problems. Another direction would be to pursue the program of providing an algebraic treatment of other families of biorthogonal functions such as those given in \cite{GM,Wil91}. The next step could be to examine the ${}_4\phi_3$ level. Finally, relating both directions, the presence of the Askey--Wilson algebra in this larger picture of bispectral biorthogonal functions would merit more investigations. We hope to study these aspects in the future.     

\paragraph{Acknowledgments.} VVC held an Undergraduate Student Research Award (USRA) from the Natural Sciences and Engineering Research Council (NSERC) of Canada. The research of LV is supported by a Discovery Grant from the NSERC. MZ holds an Alexander--Graham--Bell graduate scholarship from the NSERC.

\end{document}